\documentclass[12pt]{article}
\usepackage[T2A]{fontenc}
\usepackage[cp1251]{inputenc}
\usepackage[russian,english]{babel}
\usepackage{amsmath,amssymb,amsfonts,amsthm,bbding}

\theoremstyle{plain}
\newtheorem{theorem}{Theorem}

\newtheorem{definition}{Definition}

\newtheorem{lemma}[theorem]{Lemma}

\newtheorem{remark}{Remark}
\let\geq\geqslant
\let\leq\leqslant

\DeclareMathOperator\Sym{Sym}
\DeclareMathOperator\prob{prob}
\DeclareMathOperator\sign{sign}

\begin{document}

\title{Combinatorial and Probabilistic Formulae for Divided Symmetrization}
\author{F. Petrov}
\maketitle

\let\thefootnote\relax\footnote{
St. Petersburg Department of
V.~A.~Steklov Institute of Mathematics of
the Russian Academy of Sciences, St. Petersburg State University.
E-mail: fedyapetrov@gmail.com}

\begin{abstract}
Divided symmetrization of a function $f(x_1,\dots,x_n)$ is symmetrization of the
ratio $$DS_G(f)=\frac{f(x_1,\dots,x_n)}{\prod (x_i-x_j)},$$
where the product is taken over the set of edges of some graph $G$. 
We concentrate on the case when $G$ is a tree and $f$
is a polynomial of degree $n-1$, in this case $DS_G(f)$ is a constant
function. We give a combinatorial interpretation
of the divided symmetrization of monomials for general trees and probabilistic
game interpretation for a tree which is a path. In particular, this implies
a result by Postnikov originally proved by computing volumes
of special polytopes, and suggests its generalization.
\end{abstract}

\section{Introduction}

Let $V$ be a set of variables, $|V|=m$, say, $V=\{x_1,\dots,x_m\}$ (but further, we need and allow  sets such
as $\{x_2,x_3,x_9\}$). It is convenient to think that $V$ is well ordered: $x_1<x_2<\dots<x_m$.
For a rational function $\varphi$, with coefficients in some field,  of variables from $V$, define its symmetrization 
as $$\Sym \varphi=\sum_{\pi} \varphi(\pi_1,\dots,\pi_m),$$
where summation is taken over all $m!$ permutations $\pi$ of the variables.

Let $f$ be polynomial of degree $d$ in the variables from $V$. Then, its divided symmetrization
$$
DS(f):=\Sym \left(\frac{f}{\prod_{x,y\in V,x<y} (x-y)}\right)
$$
is also polynomial of degree not exceeding $d-m(m-1)/2$. In particular, it vanishes identically when
$d<m(m-1)/2$. The reason why $DS(f)$ is a polynomial is the following. Fix variables $x,y$ and 
partition all summands into pairs corresponding to permutations $(\pi,\sigma\pi)$, where $\sigma$ is a
transposition of $x$ and $y$. We see that in the sum of any pair, the multiple $x-y$ in the denominator gets cancelled. Thus every multiple
is cancelled and so we get polynomial. The symmetrization operators have applications, for
instance, in the theory of symmetric functions, see Chapter 7 of the A.~Lascoux's book \cite{La}.

Let $G(V,E)$ be a graph on the set of vertices $V$. We view $E$ as a set of pairs $(x,y)\in V^2$, $x<y$.
We may consider \textit{partial symmetrization in $G$},
that is,
$$
DS_G(f)=\Sym \left(\frac{f}{\prod_{(x,y) \in E} (x-y)}\right).
$$
Of course this is a polynomial again of degree at most $d-|E|$ due to the obvious formula
$$
DS_G(f)=DS\left(f\cdot \prod_{x<y,(x,y)\notin E} (x-y)\right).
$$

If we restrict $DS_G$ to polynomials of degree at most $|E|$, we get a linear functional. The kernel $K_G$ of this functional
is particularly structured. First of all, all polynomials of degree less then $d$
lie in $K_G$. Next, if $f$ has a symmetric factor, i.e., $f=gh$, where $g$ is symmetric and non-constant,
then $f\in K_G$. This is true because of the formula $DS_G(gh)=gDS_G(h)$, and the second multiple being equal to 0 since 
$\deg h<|E|$. 

Assume that $G$ is disconnected. That is, $V=U\sqcup W$, and there are no edges 
of $G$ between $U$ and $W$: $E=EU\sqcup EW$, where $EU,EW$ are sets of edges joining
vertices of $U,W$ respectively. Denote the corresponding subgraphs of $G$ by $GU=(U,EU)$ and $GW=(W,EW)$.
Note that both $U,W$ are well ordered sets of variables and thus the above definitions
still apply to the subgraphs $GU, GW$.

Any polynomial $f$ may be represented as a sum $\sum u_iw_i$, where the polynomials
$u_i$ depend only on variables from $U$, while $w_i$ depends only on variables from $W$ (and, of course,
the degree $\deg u_i+\deg w_i$ of each summand does not exceed $\deg f$). Assume that $\deg f\leq |E|$.
Then
\begin{equation}\label{eq1}
DS_G(f)=\binom{m}{|U|}\sum_i DS_{GU}(u_i)\cdot DS_{GW}(w_i)
\end{equation}
(the binomial factor comes from fixing the sets of variables $\pi(U)$ and $\pi(V)$. If 
$\deg u_i<|EU|$ then the symmetrization $DS_{GU}(u_i)$ is just 0, analogously if $\deg w_i< |EW|$. If 
$\deg u_i=|EU|$, $\deg w_i=|EW|$, then both $DS_{GU}(u_i)$, $DS_{GW}(w_i)$ are constants and
therefore do not depend on the sets of variables $\pi(U)$, $\pi(V)$).
It follows that $f\in K_G$ if for any $i$ either $u_i\in K_{GU}$ or $w_i\in K_{GW}$.
As already noted above, it is so unless $\deg u_i=|EU|$, $\deg w_i=|EW|$. 
If $f$ has a factor symmetric in the variables from $U$, then $DS_{GU}(u_i)=0$. 

Next observation. If $E'\subset E$ and $f=h\cdot \prod_{(x,y)\in E'} (x-y)$ then $DS_G(f)=DS_{G\setminus E'}(h)$.
Combining this with our previous argument, we get the following 

\begin{lemma}\label{lem1} If $E'\subset E$ and $U\subset V$ is a connected component in $G\setminus E'$, $f$ is divisible by
$h\prod_{(x,y)\in E'} (x-y)$, where $h$ is symmetric in variables from $U$, then $f\in K_G$.
\end{lemma}

Denoting by $I_G$ the set of polynomials $v$ such that $vh\in K_G$ provided that $\deg vh\leq |E|$ (it is sort of an ideal, but the
set of polynomials with restricted degree is not a ring), we have found some elements in $I_G$: all symmetric polynomials and all polynomials like those in Lemma \ref{lem1}.

Next, we consider the case of partial divided symmetrization w.r.t. tree $G$ on $n$ vertices of a polynomial $f$,
$\deg f=n-1$. This is a linear functional and we give combinatorial formulae for its values in a natural monomial base.

\section{Tree}

\begin{definition}
Let $T=(V, E)$ be a tree on a well ordered set $V$, $|V|=n$. Let
$C:=\prod_{x\in V} x^{w(x)+1}$ be a monomial of degree $n-1$,
where we call $w(x)\in \{-1,0,1,2,\dots\}$ \emph{a weight} of a vertex
$x$. The total weight of all vertices equals $-1$.
For each edge $e=(x,y)\in E$, $x<y$, 
consider two connected components of the 
graph $T\setminus e$. The total weight is negative for exactly one of them. 
If this component contains $y$, call edge $e$ \emph{regular}, else
call it \emph{inversive}. Define sign 
$\sign(C)$ as $(-1)^{\{\text{number of inversive edges}\}}$.
Call a permutation $\pi$ of the set $V$ to be \emph{$C$-acceptable} if for all edges
$e=(x,y)$, $\pi(x)<\pi(y)$
if and only if $e$ is regular. 
\end{definition}

\begin{theorem}
The partial divided symmetrization $DS_T(C)$ of the monomial $C$
equals the number of $C$-acceptable permutations times $\sign(C)$.
\end{theorem}

\begin{proof}
Induction on $n$. The base case $n=1$ is obvious. Assume that $n>1$ and the assertion is valid for $n-1$.
For any monomial $C$ denote by $\tau(C)$
the number of $C$-acceptable permutations times $\sign(C)$. We need to check that
$\tau(C)=DS_T(C)$ for all $C$. To this end,  it suffices to verify the following
properties of $\tau$ and $DS_T$:

(i) $\tau(C)-DS_T(C)$ does not depend on $C$;

(ii) $\sum_{x\in V} \tau(x^{n-1})=0=\sum DS_T(x^{n-1})$.

We start with (i). In turn, it suffices to prove that $\tau(C_1x)-DS_T(C_1x)=
\tau(C_1y)-DS_T(C_1y)$, where $C_1$ is a monomial of degree $n-2$ and
$e=(x,y)\in E$, $x<y$, is an edge of $T$. We have 
$$
DS_T(C_1x)-DS_T(C_1y)=DS_T(C_1(x-y))=DS_{T\setminus e} (C_1).
$$
Denote $V=V_x\sqcup V_y$, where $V_x$, $V_y$ are components
of $T\setminus e$ containing $x$, $y$ respectively; $T_x$, $T_y$ 
are trees induced by $T$ on $V_x$, $V_y$, 
and $C_1=C_x\cdot C_y$,
where $C_x$ is a monomial in elements the of $V_x$ and $C_y$ in the elements of $V_y$. 
We have $$
DS_{T\setminus e}(C_x\cdot C_y)=
 \begin{cases} \binom{n}{|V_x|} DS_{T_x}(C_x)\cdot DS_{T_y}(C_y)&\mbox{if } \deg C_x=|V_x|-1\\ 
0 & \mbox{otherwise }. 
\end{cases} 
$$
In the second case we also have $\tau(C_1x)=\tau(C_2x)$, since any edge $\tilde{e}$
is either regular for both monomials $C_1x$, $C_1y$, or inversive for both
monomials. In the first case edge $e$ is regular for $C_1x$ and inversive for $C_1y$. 
It means that $C_1x$ and $C_1y$ have different signs: $\sign(C_1x)=\sign(C_x)\cdot \sign(C_y)$.
It follows that $\tau(C_1x)-\tau(C_1y)$ equals $\sign(C_x)\cdot \sign(C_y)$
times number of permutations $\pi$
which are either $C_1x$-acceptable or $C_1y$-acceptable. 
If we fix $\pi(V_x)$ (and thus automatically $\pi(V_y)$), 
which may be done in exactly $\binom{n}{|V_x|}$ ways,
this property of a permutation is a combination of independent properties on
$T_x$ and on $T_y$. Applying the induction hypothesis we get that 
$\tau(C_1x)-\tau(C_1y)$ equals $\binom{n}{|V_x|} DS_{T_x}(C_x)\cdot DS_{T_y}(C_y)$,
as desired. 

Now we come to (ii). We have $\sum DS_T(x^{n-1})=DS_T(\sum x^{n-1})=0$,
since $\sum x^{n-1}$ is a symmetric polynomial. So, it suffices to prove that 
$\sum \tau(x^{n-1})=0$. This, however, is an alternating sum of numbers of pointed permutations:
$\tau(x^{n-1})$ counts permutations $\pi$ pointed by $x$ satisfying 
$\pi(u)>\pi(v)$ for any edge $uv$ so that there exists a path from $x$ to $u$ avoiding $v$.
We prove that it does vanish by providing a sign-changing 
involution on this set of pointed permutations (sign is understood in the sense of $\tau$).
Let $v_1>v_2$ be two maximal elements of $V$.
The involution 
acts as follows: take permutation $\pi$ pointed by $x$, then $\pi(x)=v_1$. Let
$y=\pi^{-1}(v_2)$, then $x$ and $y$ are neighbours. 
Replace $\pi(x)$ by $v_2$ and $\pi(y)$ by $v_1$, also point a
new permutation by $y$.  
\end{proof}

\section{Path}

Now we restrict our attention to the following specific situation. Let $m=n+1$, $V=\{x_0,\dots,x_{n}\}$,
$E(G)=\{(x_i,x_{i+1}),0 \leq i \leq n-1\}$. For notational simplicity, denote $\Phi(f)=DS_G(f)$, $K=K_G$, $I=I_G$. 
Also denote $y_i=x_0+x_1+\dots+x_i$. We have $y_n\in I$ and $y_k(x_{k+1}-x_k)\in I$ for $k=0,1,\dots,n-1$.

Let's show how this helps to calculate, say, $\Phi(x_i^n)$. 

\begin{lemma}\label{lem2}
$\Phi(x_i^n)=(-1)^i \binom{n}{i}.$
\end{lemma}

\begin{proof} Induction on $n$. The base case $n=1$ (or even $n=0$) is straightforward. 
Denote $c_i=\Phi(x_i^n)$. We have $c_0+\dots+c_n=\Phi(x_0^n+\dots+x_n^n)=0$, since symmetric polynomials
lie in $K$. Next, denoting $e=(x_k,x_{k+1})\in E$ we get
\begin{multline*}
c_k-c_{k+1}=\Phi(x_k^n-x_{k+1}^n)=DS_{G\setminus e} (x_k^{n-1}+\dots+x_{k+1}^{n-1})=\\
=DS_{G\setminus e} (x_k^kx_{k+1}^{n-k})=\binom{n+1}k (-1)^k
\end{multline*}
by \eqref{eq1} and by the induction assumption. We obtain $n+1$ linear relations on $c_0,\dots,c_n$ which determine 
them uniquely and also the numbers $c_i=(-1)^i\binom{n}{i}$ satisfy those relations. This finishes the induction step. 
\end{proof}

This lemma and many other nice formulae for $\Phi$ are obtained 
recently by T.~Amdeberhan in \cite{T}.

Now, consider the product 
$$
f(x_0,\dots,x_n)=y_0y_1\dots y_{n-1}=x_0(x_0+x_1)(x_0+x_1+x_2)\dots (x_0+\dots+x_{n-1})
$$
and proceed as follows. At first, replace $x_{n-1}$ by $x_{n-2}$ in this product. Next, replace 
$x_{n-2}$ by $x_{n-3}$, and so on. Finally we get $n! x_0^n$. The value of $\Phi$ does not change
after such replacements. Indeed, when we replace $x_{k+1}$ by $x_k$, we add (to our polynomial) a quantity divisible
by $(x_{k+1}-x_k)(x_0+\dots+x_k)$. This is an element of $I$, hence what we add lies in $K$.
Thus
\begin{equation}\label{eq2}
\Phi(y_0y_1\dots y_{n-1})=n!
\end{equation}
\begin{theorem}\label{Postnikov}
For any homogeneous polynomial $h(t_0,\dots,t_n)$ of degree $n$, we have
\begin{multline}\label{eq3}
\Phi\left(h(y_0,y_1,\dots,y_n)+h(y_1,y_2,\dots,y_n,y_0)+\dots+h(y_n,y_0,\dots,y_{n-1})\right)=\\=n! h(1,1,\dots,1).
\end{multline}
\end{theorem}

Taking $h(t_0,\dots,t_n)=(z_0t_0+\dots+z_nt_n)^n$ recovers Corollary 6.5 of \cite{P}.

\begin{proof}
It suffices to prove Theorem \ref{Postnikov} for monomials $h(t_0,\dots,t_n)=\prod t_i^{c_i}$, $\sum c_i=n$. 
We are interested only in cyclic type of $(c_0,\dots,c_{n})$, and for cyclic type $(1,1,\dots,1,0)$ this follows
from \eqref{eq2} and the fact that $y_n\in I$. Denote by $Q(c_0,\dots,c_n)$ the function of a cyclic vector 
$(c_0,\dots,c_{n})$ which is defined as the difference between two parts of \eqref{eq3}. It helps to remember
that $I$ contains $y_n$ and $y_k(x_k-x_{k+1})=y_k(2y_{k}-y_{k-1}-y_{k+1})$ (we may think that indices are taken modulo 
$n$, then it is true for all $k$). Thus we have $Q(1,1,\dots,1,0)=0$ and $2Q(c_0,c_1,\dots)-Q(c_0+1,c_1-1,c_2,\dots)-Q(c_0,c_1-1,c_2+1,\dots)=0$
if $c_1\geq 2$. These relations imply that $Q$ is always 0 by a standard maximum principle
argument. Indeed, if, say, $Q$ attains positive values and 
$Q(c_0,c_1,\dots,c_n)$ is maximal, and $c_1\geq 2$, then we see that $(c_0+1,c_1-1,c_2,\dots)$ and $(c_0,c_1-1,c_2+1,\dots)$
are other maximizers. It is easy to see that by such operations we may come to the cyclic type
$(1,1,\dots,1,0)$ and so get a contradiction.
\end{proof}

The above proof is indirect in the sense that it does not say explicitly what $\Phi(h)$ is.
This is partially fixed in the following description in terms of a probabilistic sandpile-type process.

Consider $n$ coins distributed somehow among the vertices of a regular $(n+1)$-gon, which are enumerated by $0,1,\dots,n$.
Choose any vertex $v$ containing at least 2 coins and \textit{rob} it as follows: 
take 1 coin from the vertex $v$ and put it either to the left or right neighbour
of $v$ with equal probability. Proceed as long as it is possible, i.e., until there remain no vertices containing at least 2 coins. 
Assume the process does not terminate for infinitely long. 
The number of vertices with
at least 1 coin does not decrease during our process. When it stabilizes, there is an empty vertex $v$, and its neighbour $u$ 
must be robbed finitely often, else with probability 1 $v$ becomes non-empty. Analogously consider the other (different from $v$)
neighbour of $u$
and verify that with probability 1 it must be robbed finitely often and so on. So, with probability 1 the process terminates in a finite
time. Next, the ultimate distribution of configurations does not depend on the choice of the robbed vertex at every turn. This follows from the
facts that operations commute, and any vertex with at least 2 coins must be robbed at some point before the process terminates.

There are $n+1$ possible final configurations (one empty vertex and $n$ vertices with 1 coin).

\begin{theorem}\label{sandpile} Assume that initially we have $c_i$ coins at vertex $i=0,1,\dots,n$. 
Then the probability $\prob(c_0,c_1,\dots,c_n)$ that vertex $n$ is empty in the final configuration 
equals $\Phi(\prod y_i^{c_i})/n!$.
\end{theorem}

\begin{proof} We have $\prob(1,\dots,1,0)=1$, $\prob(c_0,\dots,c_n)=0$ when $c_n>0$ and
\begin{multline*}
2\prob(c_0,c_1,\dots)=\prob(c_0,\dots,c_{i-1}+1,c_i-1,\dots)+\\+\prob(c_0,\dots,c_i-1,c_{i+1}+1,\dots)
\end{multline*}
whenever $c_i\geqslant 2$. 
The same three properties are shared by the function $\Phi(\prod y_i^{c_i})/n!$. But they define the function uniquely,
as the argument with maximum principle (very much the same as in the proof of Theorem \ref{Postnikov}) shows.
\end{proof}

\begin{remark}
By cyclic symmetry, $\prob(c_{i},c_{i+1},\dots,c_{i-1})$ is the probability that in the final configuration vertex $i-1$ is empty. 
Summing over all $i$, we get $\sum_{i=0}^n \Phi(\prod_{j=0}^n y_j^{c_{j+i}})=n!$, where summation indices are cyclic modulo $n+1$. This is another proof of Theorem \ref{Postnikov}. 
\end{remark}

We may consider a slightly more general problem. Namely, 
let ${\mathcal C}_{n+d}$ be a cycle with $n+d$ vertices
$1,\dots,n+d$ counted counter-clockwise, $x_i$ are the corresponding
variables. Indices are taken modulo $n+d$.
Put $n$ coins in the vertices, $c_i$ coins
in a vertex $i$, $\sum c_i=n$. Do the same 
robbing process as above until we get exactly $d$ empty
vertices, and 1 coin in each of the others. The sum of probabilities over all
$\binom{n+d}d$ possible final configurations equals 1, and this identity may be
rewritten as a divided symmetrization identity
due to Theorem \ref{sandpile}. 

Namely, let $P$ be a set of $d$ empty
vertices in the final configuration. Define the weight $w(P)$ as a product
of sizes of $d$ groups onto which the vertices from $P$ divide
the cycle (each vertex from $P$ belongs to exactly one group, so the
sum of sizes of these groups equals $n+d$).
For any vertex $i\notin P$ define
$z_i=z_i^P$ as a sum of variables between $i$ and the clockwise-next to $i$ vertex $p(i)$
from $P$ ($i$ included, $p(i)$ not included). Define also $z_i=1$
if $i\in P$. Then the probability that $P$ is empty in the final
configuration is
 \begin{equation}\label{nplusd}
\frac{w(P)}{(n+d)!}
DS_{\mathcal{C}_{n+d}} \prod_{p\in P} (x_{p+1}-x_{p}) \prod z_i^{c_i}.
\end{equation}
Indeed, this is clear if $c_i>0$ for some $i\in P$, both the probability that $P$ is empty
and the divided
symmetrization \eqref{nplusd} are equal to 0.
If $c_i=0$ for all  $i\in P$, the events in the $d$ groups between elements of $P$ are independent. The
symmetrizations are also independent due to the multiples
$x_{p+1}-x_p$, thus we apply \eqref{eq1} (strictly speaking, a generalization
of \eqref{eq1} for $d$ groups of independent variables.) 
The following identity generalizes Theorem \ref{Postnikov}
(which corresponds to the case $d=1$):

\begin{theorem} In above notations any polynomial $h$ of degree $n$ in $n+d$ variables
satisfies the following identity
\begin{equation}\label{cor65gen}
 \sum_P 
w(P)
DS_{\mathcal{C}_{n+d}} \prod_{p\in P} (x_{p+1}-x_{p}) h(z_1^P,\dots,z_{n+d}^P)=(n+d)!h(1,1,\dots,1)
\end{equation}
\end{theorem}
The proof is the same as in Theorem \ref{Postnikov}: it suffices to consider $h$
a monomial, in this case the result follows from \eqref{nplusd} and the fact that with probability
1 our process with coins should terminate. 

\smallskip

\textbf{Acknowledgements}
\smallskip

I am grateful to Tewodros Amdeberhan for bringing my attention to this topic, for
fruitful discussions, and for
reading and editing the preliminary version of the manuscript. 
I also thank the anonymous referee for suggesting a number
of improvements. 

The work is supported by St. Petersburg State University grant 6.37.208.2016.

\end{document}